\pdfoutput=1
\documentclass{article}

\usepackage{microtype}
\usepackage{graphicx}
\usepackage{subfigure}
\usepackage{booktabs} 

\usepackage{amsfonts}
\usepackage{math}	 
\usepackage{mathtools}  

\usepackage{natbib}

\usepackage{hyperref}

\usepackage[accepted]{icml2018}

\icmltitlerunning{Lyapunov Functions for First-Order Methods}

\newcommand{\x}{\mathbf{x}}

\newcommand{\y}{\mathbf{y}}
\newcommand{\f}{\mathbf{f}}
\newcommand{\g}{\mathbf{g}}
\newcommand{\G}{\mathbf{G}}
\newcommand{\V}{\mathcal{V}}
\newcommand{\B}{\mathbf{B}}
\newcommand{\bax}{\bar{x}}
\newcommand{\baz}{\bar{z}}
\newcommand{\bay}{\bar{y}}
\newcommand{\bag}{\bar{g}}
\newcommand{\baf}{\bar{f}}

\newcommand{\I}{\mathcal{I}}
\newcommand{\F}{\mathcal{F}_{\mu,L}}

\newcommand{\kron}{\otimes}
\newcommand{\df}{\nabla\!f}
\newcommand{\0}{\mathbf{0}}
\newcommand{\1}{\mathbf{1}}
\newcommand{\e}{\mathbf{e}}
\newcommand{\rank}{\mathrm{Rank}}
\newcommand{\real}{\mathbb{R}}
\newcommand{\nat}{\mathbb{N}}
\newcommand{\for}{\quad\textup{for }}

\DeclareMathOperator*{\feasible}{feasible}

\newcommand{\sdp}{(\textcolor{blue}{\ensuremath{\rho}-SDP})}

\usepackage[dvipsnames]{xcolor}

\definecolor{colorGM}{RGB}{55,126,184}  
\definecolor{colorHBM}{RGB}{228,26,28}  
\definecolor{colorFGM}{RGB}{152,78,163} 
\definecolor{colorTMM}{RGB}{77,175,74}  

\usepackage{tikz,pgfplots}
\usetikzlibrary{external}
\pgfplotsset{compat=1.13}

\pgfplotsset{plotOptions/.style={%
  width=\linewidth,
  y post scale=0.8,
	xlabel={Condition ratio $\kappa$},
	ylabel={Worst-case convergence rate $\rho$},
	label style={font=\small},
	legend style={font=\small},
	legend pos=south east,
	legend cell align=left,
  xtick={1,10,100,1000},
  tick label style={font=\footnotesize},
  no markers,
  solid,
  very thick}}
\pgfplotsset{MyStyle2/.style={%
		width=\linewidth,
		y post scale=0.8,
		xlabel={Condition ratio $\kappa$},
		ylabel={Evaluations to convergence $-1/\log \rho$},
		label style={font=\small},
		legend style={font=\small},
		legend pos=south east,
		legend cell align=left,
		xtick={1,10,100,1000},
		tick label style={font=\footnotesize},
		no markers,
		solid,
		very thick}}

\begin{document}
	
	\twocolumn[
	\icmltitle{Lyapunov Functions for First-Order Methods:\\ Tight Automated Convergence Guarantees}
	
	
	
	\icmlsetsymbol{equal}{*}
	
	\begin{icmlauthorlist}
		\icmlauthor{Adrien Taylor}{equal,SIERRA}
		\icmlauthor{Bryan Van Scoy}{equal,UWM}
		\icmlauthor{Laurent Lessard}{equal,UWM,ECE}
	\end{icmlauthorlist}
	
	\icmlaffiliation{UWM}{Wisconsin Institute for Discovery, University of Wisconsin--Madison, Madison, Wisconsin, USA}
	\icmlaffiliation{ECE}{Department of Elecrical and Computer Engineering, University of Wisconsin--Madison, Madison, Wisconsin, USA}
	\icmlaffiliation{SIERRA}{INRIA, D\'epartement d'informatique de l'ENS, \'Ecole normale sup\'erieure, CNRS, PSL Research University, Paris, France}
	
	\icmlcorrespondingauthor{Adrien Taylor}{adrien.taylor@inria.fr}
	\icmlcorrespondingauthor{Bryan Van Scoy}{vanscoy@wisc.edu}
	\icmlcorrespondingauthor{Laurent Lessard}{\mbox{laurent.lessard@wisc.edu}}
	
	\icmlkeywords{First-order optimization, Lyapunov function, Performance estimation problem, Integral quadratic constraints}
	
	\vskip 0.3in
	]
	
	
	
	\printAffiliationsAndNotice{\icmlEqualContribution} 
	
	\begin{abstract}		
		We present a novel way of generating Lyapunov functions for proving linear convergence rates of first-order optimization methods. Our approach provably obtains the \emph{fastest} linear convergence rate that can be verified by a quadratic Lyapunov function (with given states), and only relies on solving a small-sized semidefinite program. 	
		Our approach combines the advantages of performance estimation problems (PEP, due to \citet{drori2014}) and integral quadratic constraints (IQC, due to \citet{lessard2016}), and relies on convex interpolation (due to \citet{taylor2017smooth,taylor2017exact}).
	\end{abstract}

	\section{Introduction}
	In this work, we study first-order methods for solving the (unconstrained) minimization problem
	\begin{align}\label{Eq:prob}
	\minimize_{x\in\real^d} \ f(x) \tag{$\mathcal{P}$}
	\end{align}
	where $f : \real^d\to\real$. In the sequel, we focus on the case where $f$ is $L$-smooth and $\mu$-strongly convex, though our methodology can be adapted to a broader class of problems.
	
	To solve~\eqref{Eq:prob}, we consider methods that iteratively update their estimate of the optimizer using only gradient evaluations. One possibility for proving convergence of such methods is by finding \emph{Lyapunov functions}.
	
	A Lyapunov function can be interpreted as defining an ``energy'' that decreases geometrically with each iteration of the method, with an energy of zero corresponding to reaching the optimal solution of~\eqref{Eq:prob}. The existence of such an energy function thus provides a straightforward certificate of linear convergence for the iterative method.
	
	In this paper, we present an automated way of generating quadratic Lyapunov functions for certifying linear convergence of first-order iterative methods to solve~\eqref{Eq:prob}. The procedure relies on solving a small-sized semidefinite program (SDP) so it is computationally efficient. Moreover, the procedure is \emph{tight}, meaning that if the SDP is infeasible, then no such quadratic Lyapunov function exists.
	
	Our results unify recent SDP-based works for certifying convergence of first-order methods, namely: performance estimation problems~\cite{drori2014,taylor2017smooth} and integral quadratic constraints from robust control~\cite{lessard2016}, using smooth strongly convex interpolation~\cite{taylor2017smooth}. These connections are further discussed in Section~\ref{Sec:comparison_to_pep_iqc}.
	
	\subsection{Organization}
	The paper is organized as follows. We describe the class of methods under consideration and basic properties of Lyapunov functions in Sections~\ref{Sec:method} and~\ref{Sec:lyap?} respectively. Our main results are then presented in Section~\ref{sec:main_results}, which also features numerical examples and comparisons to other approaches. The corresponding proof is presented in Section~\ref{Sec:proof}. Finally, we explore extensions of our approach in Section~\ref{Sec:extensions}, and conclude in Section~\ref{Sec:conclusion}.
	
	\subsection{Preliminaries}
	A function $f : \real^d\to\real$ is called $L$-smooth if its gradient is Lipschitz continuous with parameter $L$, i.e.,
	\begin{align}\label{Eq:smooth}
	\|\df(x)-\df(y)\| \leq L\,\|x-y\| \quad\text{for all }x,y\in\real^d.
	\end{align}
	Furthermore, $f$ is called convex if
	\begin{align}\label{Eq:strongly_convex}
	f(x) \geq f(y) + \df(y)^\tp (x-y) \quad\text{for all }x,y\in\real^d,
	\end{align}
	and $\mu$-strongly convex if $f(x)-\tfrac{\mu}{2} \|x\|^2$ is convex. The set of $L$-smooth and $\mu$-strongly convex functions is denoted~$\F$, and we define $\kappa\defeq \frac{L}{\mu}$, the corresponding condition number.
	
	When $f\in\F$ with $0<\mu\leq L$, optimization problem~\eqref{Eq:prob} has a unique minimizer denoted ${x_\star\defeq\argmin_x f(x)}$. The function and gradient values at optimality are denoted $f_\star\defeq f(x_\star)$ and $g_\star\defeq\df(x_\star)=\0_d$, respectively.

	\section{First-Order Iterative Fixed-Step Methods}\label{Sec:method}
	To solve the optimization problem~\eqref{Eq:prob}, we consider \emph{first-order iterative fixed-step methods} of the form
	\begin{align}\label{Eq:method}\tag{$\mathcal{M}$}
	\begin{aligned}
	y_k     &= \sum_{j=0}^N \gamma_j\,x_{k-j} \\
	x_{k+1} &= \sum_{j=0}^N \beta_j\,x_{k-j} - \alpha\,\df(y_k)
	\end{aligned}
	\end{align}
	for $k\geq 0$ where $\alpha$, $\beta_j$, $\gamma_j$ are the (fixed) step-sizes and $x_j\in\real^d$ for $j=-N,\ldots,0$ are the initial conditions. We call the constant $N\geq 0$ the \emph{degree} of the method.
	
	Many first-order optimization methods are of the form~\eqref{Eq:method}, including: the Gradient Method, Heavy Ball Method~\cite{polyak1964}, Fast Gradient Method for smooth strongly convex minimization~\cite{nesterov2004}, Triple Momentum Method~\cite{vanscoy2018}, and Robust Momentum Method~\cite{cyrus2018}.
	
	For method~\eqref{Eq:method} to solve~\eqref{Eq:prob}, it must have a fixed-point at the optimizer $x_\star$. Hence, we require the step-sizes to satisfy
	\begin{align*}
	\sum_{j=0}^N \beta_j  &= 1  &  &\text{and}  &
	\sum_{j=0}^N \gamma_j &= 1.
	\end{align*}
	For simplicity, let us define the concatenated error vectors at iteration $k$ as $\x_k,\g_k\in\real^{(N+1)d}$ and $\f_k\in\real^{N+1}$ with
	\begin{subequations}\label{Eq:vectors}
		\begin{align}
		\x_k &\defeq \bmat{(x_k-x_\star)^\tp & \ldots & (x_{k-N}-x_\star)^\tp}^\tp \\
		\g_k &\defeq \bmat{(g_k-g_\star)^\tp & \ldots & (g_{k-N}-g_\star)^\tp}^\tp \\
		\f_k &\defeq \bmat{(f_k-f_\star) & \ldots & (f_{k-N}-f_\star)}^\tp
		\end{align}
	\end{subequations}
	where $x_k\in\real^d$ are the iterates, $f_k\defeq f(y_k)\in\real$ are the function values, and $g_k\defeq\df(y_k)\in\real^d$ are the gradient values. Note that we shifted $(\x_k,\g_k,\f_k)$ so that the optimal solution corresponds to $(\x_\star,\g_\star,\f_\star)=(\0,\0,\0)$.

	\section{What is a Lyapunov Function?}\label{Sec:lyap?}
	Lyapunov functions are one of the fundamental tools in control theory that can be used to verify stability of a dynamical system~\cite{kalman1960-1,kalman1960-2}.
	
	Consider applying method~\eqref{Eq:method} to solve problem~\eqref{Eq:prob}. Our goal is to find the smallest possible $0\le \rho < 1$ such that $\{x_k\}$ converges linearly to the optimizer $x_\star$ with rate $\rho$. A \textit{Lyapunov function} $\V$ is a continuous function ${\V : \real^n\to\real}$ that satisfies the following properties:
	\begin{enumerate}
		\item (nonnegative) $\V(\xi)\geq 0$ for all $\xi$,
		\item (zero at fixed-point) $\V(\xi)=0$ if and only if $\xi=\xi_\star$,
		\item (radially unbounded) $\V(\xi)\to\infty$ as $\|\xi\|\to\infty$,
		\item (decreasing) $\V(\xi_{k+1})\leq\rho^2\,\V(\xi_k)$ for $k\geq N$,
	\end{enumerate}
	where $\xi_k\defeq (\x_k,\g_k,\f_k)$ is the \emph{state} of the system at iteration $k$. The state at iteration $k$ includes past iterates, function values, and gradient values from iterations $k-N$ up to $k$.
	If we can find such a $\V$, then it can be used to show that the state converges linearly to the fixed-point from any initial condition (the rate of convergence depends on both $\rho$ and the structure of $\V$).
	
	Lyapunov functions are typically found by searching over a parameterized family of functions (called Lyapunov function candidates). In the simple case where the state $\{\xi_k\}$ is generated by a linear dynamical system, one can search over quadratic Lyapunov function candidates by solving a semidefinite program, as illustrated in Example~\ref{Ex:quad_lyap_fun} below.

	\begin{ex}[Quadratic Lyapunov function]\label{Ex:quad_lyap_fun}
		Consider the linear dynamical system described by
		\begin{align*}
		\xi_{k+1} = A \xi_k, \qquad \xi_0\in\real^n
		\end{align*}
		with fixed-point $\xi_\star\in\real^n$ (i.e., $\xi_\star = A \xi_\star$). Suppose that
		\begin{align}\label{Eq:linear_system_SDP}
		\smash[b]{\feasible_{P\in\mathbb{S}^n}} \quad 0\succeq A^\tp P A - \rho^2 P, \quad P\succ 0
		\end{align}
		has solution $P_\star$. Then a Lyapunov function for the system is
		\begin{align*}
		\V(\xi) = (\xi-\xi_\star)^\tp P_\star (\xi-\xi_\star)
		\end{align*}
		which can be used to show that $\xi_k\to \xi_\star$ linearly with rate~$\rho$. Specifically, we have the bound
		\[
		\|\xi_k-\xi_\star\|_{P_\star}\leq \rho^k \,\|\xi_0-\xi_\star\|_{P_\star} \for k\geq 0.
		\]
		To find the best bound, we can perform a bisection search on $\rho$ to find the smallest $\rho$ such that~\eqref{Eq:linear_system_SDP} is feasible.
	\end{ex}
	
	Note that although $\V$ depends explicitly on the fixed point $\xi_\star$, we do not need to know $\xi_\star$ to solve the SDP~\eqref{Eq:linear_system_SDP}.
	
	The linear dynamical system of Example~\ref{Ex:quad_lyap_fun} converges linearly if and only if a quadratic Lyapunov function exists, which happens if and only if the SDP~\eqref{Eq:linear_system_SDP} is feasible~\cite{lyapunov,vidyasagar2002}.

	\section{Main Results}\label{sec:main_results}
	
	Similar to Example~\ref{Ex:quad_lyap_fun}, we now show how to use quadratic Lyapunov functions to prove linear convergence of a first-order iterative fixed-step method applied to the minimization of a smooth strongly convex function. Furthermore, we show that such Lyapunov function exists if and only if a small-sized semidefinite program is feasible (whose optimal solution produces the Lyapunov function). 
	
	\subsection{Quadratic Lyapunov Functions}
	We begin with sufficiency: if we can find a quadratic Lyapunov function, we can use it to prove linear convergence.
	
	\begin{lem}[Quadratic Lyapunov function]\label{Lem:lyap}
		Consider applying the first-order iterative fixed-step method~\eqref{Eq:method} of degree $N$ to a smooth strongly convex function ${f\in\F(\real^d)}$ with ${0<\mu\leq L}$. Define the state $\xi_k\defeq (\x_k,\g_k,\f_k)$  as in~\eqref{Eq:vectors}. Consider the quadratic function
		\begin{align}\label{Eq:lyap}
		\V(\xi_k) = \bmat{\x_k \\ \g_k}^\tp\!\!\! (P\otimes I_d) \bmat{\x_k \\ \g_k} + p^\tp \f_k\quad\text{for $k \ge N$}
		\end{align}
		with parameters $P\in\mathbb{S}^{2(N+1)}$ and $p\in\real^{N+1}$, and where~$\otimes$ denotes the Kronecker product. Suppose $\V$ is a Lyapunov function for the system with rate $\rho$. Then, the following bound is satisfied:
		\begin{equation}\label{Eq:bound}
		\V(\xi_k) \leq \rho^{2(k-N)}\,\V(\xi_N) \for k\geq N.
		\end{equation}
	\end{lem}
	\begin{proof}
		Suppose $\V$ is a Lyapunov function for method~\eqref{Eq:method} with $f\in\F$. Then $0\geq\V(\xi_{i+1})-\rho^2\,\V(\xi_i)$ for $i\geq N$. Multiplying this inequality by $\rho^{2(k-i-1)}$ and summing over $i=N,\ldots,k-1$ gives a telescoping sum that yields~\eqref{Eq:bound}.
	\end{proof}
	As a consequence to Lemma~\ref{Lem:lyap}, we have the relations:
	\begin{subequations}\label{Eq:rate}
		\begin{align}
		\|x_k-x_\star\| &= \mathcal{O}(\rho^k) \\
		\|\df(y_k)\|    &= \mathcal{O}(\rho^k) \\
		f(y_k)-f_\star  &= \mathcal{O}(\rho^{2k})
		\end{align}
	\end{subequations}
	where $x_\star\in\real^d$ is the optimizer of~\eqref{Eq:prob} and $f_\star\defeq f(x_\star)$.
	
	\begin{rem}
	  The Lyapunov function~\eqref{Eq:lyap} is only defined for $k\geq N$ since the state $\xi_k$ is a function of the previous $N$ function and gradient values. This is why the bound~\eqref{Eq:bound} is expressed in terms of $\V(\xi_N)$.
	\end{rem}
	\begin{rem}
		The states used in the Lyapunov function~\eqref{Eq:lyap} can be modified to  include other iterates (such as~$y_k$) in the quadratic term as well as the function and gradient values evaluated at iterates other than~$y_k$. We chose the form in~\eqref{Eq:lyap} because it contains all necessary ingredients while also being straightforward to generalize to other cases.
		
		In addition, note that the structure of~\eqref{Eq:lyap} makes it \emph{permutation-invariant} (i.e., it does not depend on the ordering of the coordinate set). This is largely motivated by the fact that there is no reason to favor any coordinate among $\real^d$.
	\end{rem}
	Lemma~\ref{Lem:lyap} shows that if we can find a quadratic Lyapunov function, then we can use this to prove linear convergence of method~\eqref{Eq:method} when $f\in\F$. In the following section, we construct an SDP whose feasibility is necessary and sufficient for the existence of such a Lyapunov function.
	
	\subsection{SDP for Quadratic Lyapunov Functions}\label{Sec:sdp}
	
	Given parameters $\alpha$, $\beta_j$, and $\gamma_j$ for a method~\eqref{Eq:method} of degree $N$ and a rate $\rho$ to be verified, we construct the semidefinite program as follows.
	
	\paragraph{Step 1: Initialization.} First, we initialize the row vectors $\bar{x}_k^{(K)},\bar{g}_k^{(K)}\in\real^{N+K+2}$ and $\bar{f}_k^{(K)}\in\real^{K+1}$, corresponding to the initial conditions, gradient values, and  function values, respectively, as
	\begin{subequations}\label{Eq:xgf_vectors}
		\begin{alignat}{2}
		&\bar{x}_k^{(K)} \defeq \mathbf{e}_{k+N+1}^\tp &&\for k\in\{-N,\ldots,0\} \\
		&\bar{g}_k^{(K)} \defeq \mathbf{e}_{k+N+2}^\tp &&\for k\in\{0,\ldots,K\} \\
		&\bar{f}_k^{(K)} \defeq \mathbf{e}_{k+1}^\tp   &&\for k\in\{0,\ldots,K\}
		\end{alignat}
	\end{subequations}
	for $K\in\{N,N+1\}$ ($\mathbf{e}_i$ denotes the $i^\text{th}$ unit vector with appropriate dimension). These form a basis for all iterates, function values, and gradient values up to iteration~$K$. Also, define the row vectors corresponding to the fixed-point as
	\begin{align*}
	\bar{y}_\star^{(K)} &\defeq \0_{N+K+2}^\tp,  &
	\bar{g}_\star^{(K)} &\defeq \0_{N+K+2}^\tp,  &
	\bar{f}_\star^{(K)} &\defeq \0_{K+1}^\tp.
	\end{align*}
	We also introduce the following SDP variables:
	\begin{align*}
	P &\in\mathbb{S}^{2(N+1)}, & \lambda_{ij} &\in\real\for i,j\in\I_N, \\
	p &\in\real^{N+1}, & \eta_{ij}    &\in\real\for i,j\in\I_{N+1},
	\end{align*}
	where $\I_K \defeq\{0,1,\dots,K,\star\}$ is an index set.
	
	\paragraph{Step 2: Method.} Next, we iterate the method for ${k=0,\ldots,K}$ using the row vectors we previously defined.
	\begin{subequations}\label{Eq:method_vectors}
		\begin{align}
		\bar{y}_k^{(K)}     &= \sum_{j=0}^N \gamma_j\,\bar{x}_{k-j}^{(K)} \\
		\bar{x}_{k+1}^{(K)} &= \sum_{j=0}^N \beta_j\,\bar{x}_{k-j}^{(K)} - \alpha\,\bar{g}_k^{(K)}.
		\end{align}
	\end{subequations}
	
	\paragraph{Step 3: Interpolation conditions\footnote{The terms $M_{ij}^{(K)}$ and $m_{ij}^{(K)}$ are related to \emph{interpolation} by smooth strongly convex functions as discussed in Section~\ref{Sec:interp}}.} Using the computed vectors, define $m_{ij}^{(K)}\in\real^{K+1}$ and $M_{ij}^{(K)}\in\mathbb{S}^{N+K+2}$ as
	\begin{subequations}\label{Eq:Mij}
		\begin{align}
		m_{ij}^{(K)} &\defeq (L-\mu) \bigl(\bar{f}_i^{(K)}-\bar{f}_j^{(K)}\bigr)^\tp \\
		M_{ij}^{(K)} &\defeq \frac{1}{2}
		\bmat{\bar{y}_i^{(K)}\\[1mm] \bar{y}_j^{(K)}\\[1mm] \bar{g}_i^{(K)}\\[1mm] \bar{g}_j^{(K)}}^\tp \!\!\!\! M
		\bmat{\bar{y}_i^{(K)}\\[1mm] \bar{y}_j^{(K)}\\[1mm] \bar{g}_i^{(K)}\\[1mm] \bar{g}_j^{(K)}}
		\end{align}
	\end{subequations}
	for $i,j\in\I_K$ where
	\begin{align}\label{Eq:M}
	M \defeq \begin{bmatrix*}[r]
	-\mu L  & \mu L & \mu  & -L \\
	\mu L  & -\mu L & -\mu  & L \\
	\mu  & -\mu & -1 & 1 \\
	-L & L & 1 & -1 \end{bmatrix*}\!.
	\end{align}	
	
	\paragraph{Step 4: Lyapunov function.} We now construct the linear and quadratic terms in the Lyapunov function, denoted ${v_k^{(K)}\in\real^{K+1}}$ and ${V_k^{(K)}\in\mathbb{S}^{N+K+2}}$, respectively, as
	\begin{subequations}\label{Eq:Vk}
		\begin{align}
		v_k^{(K)} &\defeq p^\tp\, \bar{\f}_k^{(K)} \\
		V_k^{(K)} &\defeq \bmat{\bar{\x}_k^{(K)} \\[1mm] \bar{\g}_k^{(K)}}^\tp\!\! P \bmat{\bar{\x}_k^{(K)} \\[1mm] \bar{\g}_k^{(K)}}
		\end{align}
	\end{subequations}
	where the matrices $\bar{\x}_k^{(K)},\bar{\g}_k^{(K)}\in\real^{(N+1)\times(N+K+2)}$ and ${\bar{\f}_k^{(K)}\in\real^{(N+1)\times(K+1)}}$ are defined as
	\begin{align*}
	\bar{\x}_k^{(K)} \!&\defeq \bmat{\bar{x}_{k}^{(K)} \\ \vdots \\ \bar{x}_{k-N}^{(K)}}  &
	\bar{\g}_k^{(K)} \!&\defeq \bmat{\bar{g}_{k}^{(K)} \\ \vdots \\ \bar{g}_{k-N}^{(K)}}  &
	\bar{\f}_k^{(K)} \!&\defeq \bmat{\bar{f}_{k}^{(K)} \\ \vdots \\ \bar{f}_{k-N}^{(K)}}\!.
	\end{align*}
	Also, define the decrease in the linear and quadratic terms of the Lyapunov function as
	\begin{subequations}\label{Eq:dV}
		\begin{align}
		\Delta v_k^{(K)} &\defeq v_{k+1}^{(K)} - \rho^2\,v_k^{(K)} \\
		\Delta V_k^{(K)} &\defeq V_{k+1}^{(K)} - \rho^2\,V_k^{(K)}
		\end{align}
		where $\rho$ is the convergence rate to be verified.
	\end{subequations}
	
	\paragraph{Step 5: Semidefinite program.} Finally, we compute the quadratic Lyapunov function (if one exists) for a given rate~$\rho$ by solving the following semidefinite program:
	
	\medskip
	\fbox{
		\begin{minipage}{0.96\columnwidth}
			\textbf{SDP for quadratic Lyapunov function \sdp}
			\begin{align*}
			\smash[b]{\feasible_{\substack{P\in\mathbb{S}^{2(N+1)}\\p\in\real^{N+1}\\ \{\lambda_{ij}\}\\ \{\eta_{ij}\}}}}
			& \quad 0 \prec V_N^{(N)} - \sum_{i,j\in\I_N} \lambda_{ij}\,M_{ij}^{(N)} \\
			& \quad 0 < v_N^{(N)} - \sum_{i,j\in\I_N} \lambda_{ij}\,m_{ij}^{(N)} \\
			& \quad 0 \succeq \Delta V_N^{(N+1)} + \!\!\!\sum_{i,j\in\I_{N+1}} \!\!\eta_{ij}\,M_{ij}^{(N+1)} \\
			& \quad 0 \geq \Delta v_N^{(N+1)} + \!\!\!\sum_{i,j\in\I_{N+1}} \!\! \eta_{ij}\,m_{ij}^{(N+1)} \\
			& \quad 0 \leq \lambda_{ij} \for i,j\in\I_N \\
			& \quad 0 \leq \eta_{ij} \for i,j\in\I_{N+1}
			\end{align*}
		\end{minipage}}
		\smallskip
		\newpage
		\begin{thm}[Main Result]\label{Thm:main}
			Consider applying the first-order iterative fixed-step method~\eqref{Eq:method} of degree $N$ to a smooth strongly convex function ${f\in\F(\real^d)}$ with ${0<\mu\leq L}$. Let the step-sizes $\alpha$, $\beta_j$, and $\gamma_j$ be such that $\alpha\neq 0$, $\gamma_0\neq 0$, and
			\[
			\sum_{j=0}^N \beta_j = \sum_{j=0}^N \gamma_j = 1.
			\]
			Then there exists a quadratic Lyapunov function of the form~\eqref{Eq:lyap} with rate $\rho$ that is valid for all $d\in\nat$ if and only if \textup{\sdp} is feasible.
		\end{thm}
		
		From Theorem~\ref{Thm:main}, we can perform bisection on~$\rho$ to find the minimum~$\rho$ such that \sdp\ is feasible to produce the \emph{fastest} linear convergence rate that is able to be verified using a quadratic Lyapunov function with states $(\x,\g,\f)$.
		
		\subsection{Comparison to PEP and IQC}\label{Sec:comparison_to_pep_iqc}
		Our results are closely related to several other recent approaches utilizing semidefinite programs for studying convergence of first-order methods, which we discuss now.
		
		\paragraph{Performance Estimation Problem (PEP).}
		The performance estimation approach was introduced by~\citet{drori2014} as a systematic way to obtain worst-case performance guarantees of a given method. In the context of fixed-step first-order methods, performance estimation problems (PEP) can be formulated as semidefinite programs.
		
		The key idea in PEP is to look for a tuple $(x_{-N},\hdots,x_0,f)$ such that the given algorithm behaves in the worst possible way, according to a given performance measure. The dual of PEP corresponding to the performance measure 
		$\V(\xi_{N+1})/\V(\xi_N)$ for some fixed $P$ and $p$ is exactly the same as solving $\min_\rho \rho^2$ subject to~\sdp\ being feasible.
		
		The difference between PEP and our approach is that in PEP, the optimization is for a fixed performance measure carried out over multiple timesteps. This yields exact worst-case bounds, but at the cost of solving an SDP whose size is proportional to the number of timesteps (this allows, among others, dealing with time-varying methods and sublinear convergence rates). In our approach, for a fixed $\rho$, we optimize \textit{the performance measure itself}. This yields a Lyapunov function with a guaranteed decrease at every iteration while (1) maintaining tightness and (2) solving a small SDP of fixed size.
		Both approaches ensure tightness via \emph{(smooth) convex interpolation}, developed by~\citet{taylor2017smooth}.

		\paragraph{Integral Quadratic Constraints (IQCs).}
		Integral quadratic constraints are an analysis method for bounding the worst-case performance of dynamical systems in feedback with nonlinearities~\cite{megretski1997}. This approach was recently adapted for use in analyzing first-order optimization algorithms~\cite{lessard2016}.
		In the optimization context, the nonlinear component is the gradient of the objective function, while the dynamical system is the iterative method being analyzed. 
		
		The key idea with IQCs is to replace the nonlinearity ($\df$) by quadratic constraints that it must satisfy. 
		This is precisely the idea behind \emph{interpolation} (discussed in Section~\ref{Sec:interp}), which is a foundational concept in our methodology. 
		
		The difference between IQCs and our approach is that the interpolation conditions are necessary and sufficient to characterize $\df$ when $f\in\F$. However, the sector IQC and weighted off-by-one IQC used by~\citet{lessard2016} are a strict subset of the interpolation conditions; they are only sufficient for describing $\df$ when $f\in\F$.
		In particular, the IQC framework does not use any constraints on $\df$ that explicitly involve function values. This amounts to solving \sdp\ with additional constraints on $\lambda_{ij}$ and $\eta_{ij}$ such that all the function values cancel out in the SDP.
		
		\subsection{Numerical Comparisons}
		To illustrate our results, we consider the Gradient Method (GM), Heavy Ball Method (HBM), Fast Gradient Method (FGM), and Triple Momentum Method (TMM). Each of these methods can be parametrized as
		\begin{subequations}\label{Eq:algo_short}
		\begin{align}
		y_k&=x_k+\gamma\, (x_k-x_{k-1})\\
		x_{k+1}&=x_k+\beta\, (x_k-x_{k-1})-\alpha\, \df(y_k)
		\end{align}
		\end{subequations}
		for $k\geq 0$ where $x_{-1},x_0\in\real^d$ are the initial conditions, and the 
		parameters for each method are:
		\begin{equation*}
		 \renewcommand{\arraystretch}{1.5}
		\begin{array}{l|ccc}
		\text{Method} & \alpha & \beta & \gamma\\ \hline
		\text{GM} & \frac1L & 0 & 0\\
		\text{HBM} & \frac{4}{(\sqrt{L}+\sqrt{\mu})^2} & \Bigl(\frac{\sqrt{\kappa}-1}{\sqrt{\kappa}+1}\Bigr)^2 & 0\\
		\text{FGM} & \frac1L & \frac{\sqrt{\kappa}-1}{\sqrt{\kappa}+1} & \frac{\sqrt{\kappa}-1}{\sqrt{\kappa}+1} \vphantom{\Bigr)^2}\\
		\text{TMM} & \frac{2\sqrt{L}-\sqrt{\mu}}{L\sqrt{L}} & \frac{(\sqrt\kappa-1)^2}{\kappa+\sqrt\kappa} & \frac{(\sqrt\kappa-1)^2}{2\kappa+\sqrt\kappa-1} \vphantom{\Bigr)^2}
		\end{array}
		\end{equation*}
		We use~\sdp\ to find corresponding Lyapunov functions. The corresponding convergence rates are provided in Figure~\ref{fig:GM_HBM_FGM_TMM}; the results match those obtained using IQCs~\cite{lessard2016} for GM, HBM, and FGM, and those for TMM provided in~\cite{vanscoy2018}. For more complicated cases, the performance estimation toolbox PESTO~\cite{pesto2017} can be used to perform numerical validations.
		
		For illustrative purposes, we present results obtained using a restricted class of Lyapunov functions. We fixed $\lambda_{ij}=0$ in~\sdp\ and plotted the best achievable $\rho$ in Figure~\ref{fig:FGM_TMM}.
		 We observe that this restricted class is not sufficient to recover the rates obtained in Figure~\ref{fig:GM_HBM_FGM_TMM}.
		
		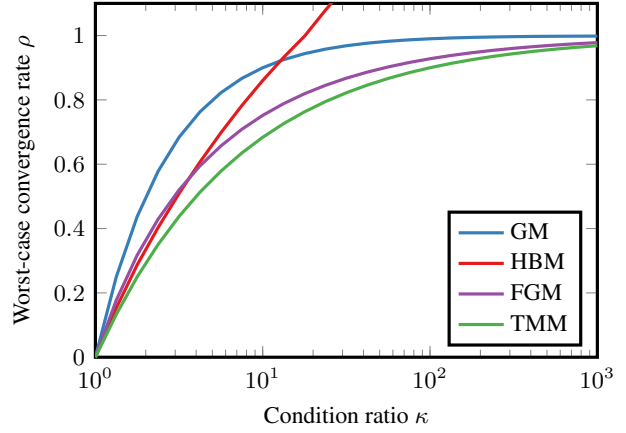
\begin{figure}[!ht]		
			\begin{center}
				\begin{tikzpicture}
				  \begin{semilogxaxis}[plotOptions,xmin=1e0,xmax=1e3,ymin=0,ymax=1.1,ytick={0,0.2,0.4,0.6,0.8,1},y post scale=0.85]
				    \addplot [colorGM]  table [x index=0,y index=1,header=false] {data/data.dat}; \addlegendentry{GM};
				    \addplot [colorHBM] table [x index=0,y index=2,header=false] {data/data.dat}; \addlegendentry{HBM};
				    \addplot [colorFGM] table [x index=0,y index=3,header=false] {data/data.dat}; \addlegendentry{FGM};
				    \addplot [colorTMM] table [x index=0,y index=4,header=false] {data/data.dat}; \addlegendentry{TMM};
				  \end{semilogxaxis}
				\end{tikzpicture}
				\caption{\label{fig:GM_HBM_FGM_TMM} Worst-case linear convergence rates from \sdp.}
			\end{center}
		\end{figure}
		
		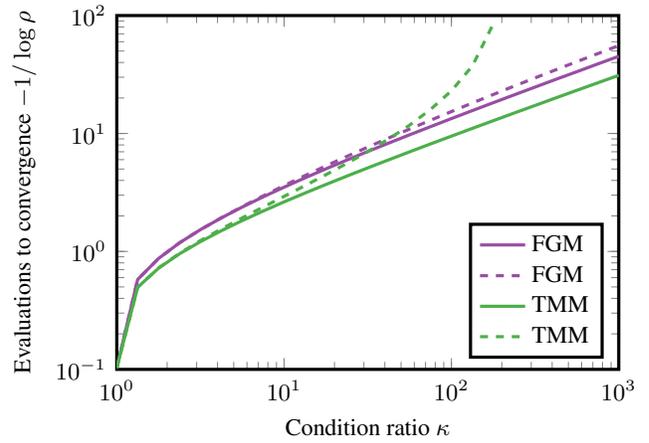
\begin{figure}[!ht]		
			\begin{center}
				\begin{tikzpicture}
				  \begin{loglogaxis}[MyStyle2,xmin=1e0,xmax=1e3,ymin=1e-1,ymax=1e2,y post scale=0.85,]
				    \addplot [colorFGM] table [x index=0,y index=1,header=false] {data/data_PositiveDefP.dat}; \addlegendentry{FGM};
				    \addplot [colorFGM,dashed] table [x index=0,y index=2,header=false] {data/data_PositiveDefP.dat}; \addlegendentry{FGM};
				    \addplot [colorTMM] table [x index=0,y index=3,header=false] {data/data_PositiveDefP.dat}; \addlegendentry{TMM};
				    \addplot [colorTMM,dashed] table [x index=0,y index=4,header=false] {data/data_PositiveDefP.dat}; \addlegendentry{TMM};
				  \end{loglogaxis}
				\end{tikzpicture}
				\caption{\label{fig:FGM_TMM} Order of magnitude of the worst-case number of iterations, which is $\mathcal{O}(-1/\log\rho)$, to solve problem~\eqref{Eq:prob}. The bounds are obtained by searching for Lyapunov functions of the form~\eqref{Eq:lyap} in two cases: (i) $(P,p)$ found using \sdp\ (solid), and (ii) restricting $P\succ 0$ and $p>0$ (dashed).}
			\end{center}
		\end{figure}
		
		\section{Proof of Theorem~\ref{Thm:main}}\label{Sec:proof}
		
		\subsection{Sampled Smooth Strongly Convex Functions}\label{Sec:interp}
		To prove Theorem~\ref{Thm:main}, we first need a result on smooth strongly convex functions that are sampled at discrete points. Indeed, inequalities~\eqref{Eq:smooth} and~\eqref{Eq:strongly_convex} completely characterize functions that are smooth and strongly convex. However, these inequalities are defined on an infinite set of points, and it was shown in Section 2.2 of~\cite{taylor2017smooth} that using them to prove convergence may introduce conservatism. Therefore, in order to completely characterize points which are sampled from smooth strongly convex functions, we need the concept of \emph{interpolation}.
		
		The following theorem is borrowed from~\cite{taylor2017smooth} and forms the basic building block for our analysis.
		
		\begin{thm}[$\F$-interpolation]\label{Thm:interp}
			Let $\I$ be an index set, and consider the set of triples $S = \{(y_i,g_i,f_i)\}_{i\in\I}$ where $y_i,g_i\in\real^d$ and $f_i\in\real$ for all $i\in\I$. There exists a function\footnote{In other words, we say that the set $S$ is \emph{$\F$-interpolable}.} $f\in\F$ such that $f(y_i)=f_i$ and $\df(y_i)=g_i$ for all $i\in\I$ if and only if $\phi_{ij}\geq 0$ for all $i,j\in\I$ where
			\begin{equation}\label{Eq:phi}
			\phi_{ij} \defeq (L-\mu) (f_i-f_j) + \bmat{y_i\\y_j\\g_i\\g_j}^\tp\!\!\!(M\kron I_d) \bmat{y_i\\y_j\\g_i\\g_j}
			\end{equation}
			with $M\in\mathbb{S}^4$ defined in~\eqref{Eq:M}.
		\end{thm}
		
		\subsection{Positive Definite Quadratics From Sampling}
		Recall from Section~\ref{Sec:lyap?} that the Lyapunov function must satisfy two conditions: (i) $\V$ must be positive definite (i.e, nonnegative, zero at the fixed-point, and radially unbounded), and (ii) $\Delta\V\defeq\V_{k+1}-\rho^2\,\V_k$ must be negative semidefinite (i.e., $\V$ must satisfy the decrease condition). To prove both (i) and (ii), we use the following theorem, which provides necessary and sufficient conditions for a quadratic form to be positive (semi-)definite when the iterates are generated by method~\eqref{Eq:method} applied to $f\in\F$.
		
		\begin{thm}[Sampled positive definite quadratics]\label{Thm:PD_quad}
			Consider applying the first-order iterative fixed-step method~\eqref{Eq:method} of degree $N$ to a smooth strongly convex function ${f\in\F(\real^d)}$ for $K$ iterations. Suppose the step-sizes $\alpha$, $\beta_j$, and $\gamma_j$ are such that $\alpha\neq 0$, $\gamma_0\neq 0$, and
			\[
			\sum_{j=0}^N \beta_j = \sum_{j=0}^N \gamma_j = 1.
			\]
			Define the vectors $\x\in\real^{(N+1)d}$, $\g\in\real^{(K+1)d}$, and ${\f\in\real^{K+1}}$ as
			\begin{subequations}\label{Eq:xgf}
				\begin{align}
				\x &\defeq \bmat{(x_{-N}-x_\star)^\tp & \ldots & (x_0-x_\star)^\tp}^\tp \\
				\g &\defeq \bmat{(g_0-g_\star)^\tp & \ldots & (g_K-g_\star)^\tp}^\tp \\
				\f &\defeq \bmat{f_0-f_\star & \ldots & f_K-f_\star}^\tp
				\end{align}
			\end{subequations}
			and denote the triple $\xi\defeq (\x,\g,\f)$. Define $m_{ij}^{(K)}\in\real^{K+1}$ and $M_{ij}^{(K)}\in\mathbb{S}^{N+K+2}$ such that
			\begin{align}\label{Eq:phi_global}
			\phi_{ij}(\xi) &= \bmat{\x \\ \g}^\tp\!\!\! (M_{ij}^{(K)}\kron I_d) \bmat{\x \\ \g} + \bigl(m_{ij}^{(K)}\bigr)^\tp \f
			\end{align}
			for $i,j\in\I_K\defeq\{0,\ldots,K,\star\}$ where $\phi_{ij}$ is defined in~\eqref{Eq:phi}. Consider the quadratic function
			\begin{align*}
			\sigma(\xi) &= \bmat{\x \\ \g}^\tp\!\!\! (Q\kron I_d) \bmat{\x \\ \g} + q^\tp \f
			\end{align*}
			where $Q\in\mathbb{S}^{N+K+2}$ and $q\in\real^{N+1}$. Suppose the dimension $d$ satisfies\footnote{This requirement is only used for necessity.} $d\geq N+K+2$.
			
			Then $\sigma$ is positive semidefinite (i.e., nonnegative) if and only if there exists $\tau_{ij}\geq 0$ for $i,j\in\I_K$ such that
			\begin{subequations}
				\begin{align*}
				0 &\preceq Q - \sum_{i,j\in\I_K} \tau_{ij}\,M_{ij}^{(K)} \\
				0 &\leq    q - \sum_{i,j\in\I_K} \tau_{ij}\,m_{ij}^{(K)}.
				\end{align*}
			\end{subequations}
			Furthermore, $\sigma$ is positive definite if and only if there exists $\tau_{ij}\geq 0$ for $i,j\in\I_K$ such that
			\begin{subequations}\label{Eq:S2_strict}
				\begin{align}
				0 &\prec Q - \sum_{i,j\in\I_K} \tau_{ij}\,M_{ij}^{(K)}  \label{Eq:S2_P_strict}\\
				0 &<     q - \sum_{i,j\in\I_K} \tau_{ij}\,m_{ij}^{(K)}. \label{Eq:S2_p_strict}
				\end{align}
			\end{subequations}
		\end{thm}
		
		\begin{proof}
			We prove the second statement that $\sigma$ is positive definite if and only if there exists $\tau_{ij}\geq 0$ for ${i,j\in\I_K}$ such that~\eqref{Eq:S2_strict} holds; the proof of the first statement is similar.
			
			Here we prove that the conditions are sufficient for $\sigma$ to be positive definite; necessity is more involved and can be found in the supplementary material.
			
			\textbf{(Sufficiency).}
			Suppose there exists $\tau_{ij}\geq 0$ for $i,j\in\I_K$ such that~\eqref{Eq:S2_strict} holds. Clearly, we have $\sigma(0)=0$. Now assume that $\xi\neq 0$. Sum the following two inequalities: (i) take the Kronecker product of~\eqref{Eq:S2_P_strict} with $I_d$ and multiply the result on the left and right by $\bmat{\x^\tp & \g^\tp}$ and its transpose, respectively, and (ii) multiply the tranpose of~\eqref{Eq:S2_p_strict} on the right by $\f$. Doing so gives the inequality
			\begin{align}\label{Eq:PD_quad_sufficiency}
			0 < \sigma(\xi) - \sum_{i,j\in\I_K} \tau_{ij}\,\phi_{ij}(\xi)
			\end{align}
			which is strict due to the strict inequalities in~\eqref{Eq:S2_strict} and since $\xi\neq 0$. Since $f\in\F$, we have $\phi_{ij}\geq 0$ from Thm.~\ref{Thm:interp}, so
			\begin{align*}
			0 \leq \sum_{i,j\in\I_K} \tau_{ij}\,\phi_{ij}(\xi) < \sigma(\xi).
			\end{align*}
			Then $\sigma(\xi)\geq 0$, and $\sigma(\xi)=0$ if and only if $\xi=0$. Finally, note that the strict inequalities in~\eqref{Eq:S2_strict} imply that the right side of~\eqref{Eq:PD_quad_sufficiency} grows arbitrarily large as $\|\xi\|\to\infty$, so $\sigma$ is radially unbounded. Thus, $\sigma$ is positive definite.
		\end{proof}
		
		\begin{rem}
			Theorem~\ref{Thm:PD_quad} can be seen as a specialized application of the S-procedure~\cite{LMI1994,megretski1993} where the points in $\xi$ are generated by method~\eqref{Eq:method}, and the positive semidefinite quadratic terms come from the interpolation conditions in Theorem~\ref{Thm:interp}. While the S-procedure is known to be lossy in certain cases (i.e., the conditions are sufficient but not necessary for $\sigma$ to be positive (semi)definite), Theorem~\ref{Thm:PD_quad} shows that it is in fact \emph{lossless} under the large-scale assumption $d\geq N+K+2$.
		\end{rem}
		
		We now apply Theorem~\ref{Thm:PD_quad} to obtain necessary and sufficient conditions for both $\V$ to be positive definite and $\Delta\V\defeq\V(\xi_{k+1})-\rho^2\,\V(\xi_k)$ to be negative semidefinite. To that end, note that the basis vectors $\bar{x}_k^{(K)}$, $\bar{g}_k^{(K)}$, and $\bar{f}_k^{(K)}$ in~\eqref{Eq:xgf_vectors} are such that
		\begin{alignat*}{3}
		& x_k-x_\star &&= (\bar{x}_k^{(K)}\kron I_d) \bmat{\x \\ \g} && \for k\in\{-N,\ldots,K\} \\
		& g_k-g_\star &&= (\bar{g}_k^{(K)}\kron I_d) \bmat{\x \\ \g} && \for k\in\{0,\ldots,K\} \\
		& f_k-f_\star &&= \bar{f}_k^{(K)} \f                         && \for k\in\{0,\ldots,K\} \\
		& y_k-x_\star &&= (\bar{y}_k^{(K)}\kron I_d) \bmat{\x \\ \g} && \for k\in\{0,\ldots,K\}. 
		\end{alignat*}
		where $\x$, $\g$, and $\f$ are defined in~\eqref{Eq:xgf} and we used the iterations in~\eqref{Eq:method_vectors}. We can then sum the following: (i) take the Kronecker product of $M_{ij}^{(K)}$ in~\eqref{Eq:Mij} with $I_d$ and multiply the result on the left and right by $\bmat{\x^\tp & \g^\tp}$ and its transpose, respectively, and (ii) multiply the transpose of $m_{ij}^{(K)}$ on the right by $\f$. Adding these two quantities gives~\eqref{Eq:phi_global}. Similarly, the Lyapunov function in~\eqref{Eq:lyap} is given by
		\begin{align}\label{Eq:V_global}
		\V(\xi_k) &= \bmat{\x \\ \g}^\tp\!\!\! (V_k^{(K)}\kron I_d) \bmat{\x \\ \g} + \bigl(v_k^{(K)}\bigr)^\tp \f
		\end{align}
		and the decrease in the Lyapunov function is given by
		\begin{align}\label{Eq:dV_global}
		\Delta V(\xi_k) &= \bmat{\x \\ \g}^\tp\!\!\! (\Delta V_k^{(K)}\kron I_d) \bmat{\x \\ \g} + \bigl(\Delta v_k^{(K)}\bigr)^\tp \f
		\end{align}
		using the definitions in~\eqref{Eq:Vk} and~\eqref{Eq:dV}. 	This leads to the following results.
		
		\begin{cor}[$\V$ positive definite]\label{Cor:V_PD}
			$\V$ in~\eqref{Eq:lyap} is positive definite for all values of $d\in\nat$ if and only if there exists $\lambda_{ij}\geq 0$ for $i,j\in\I_N$ such that
			\begin{align*}
			0 &\prec V_N^{(N)} - \sum_{i,j\in\I_N} \lambda_{ij}\,M_{ij}^{(N)} \\
			0 &<     v_N^{(N)} - \sum_{i,j\in\I_N} \lambda_{ij}\,m_{ij}^{(N)}
			\end{align*}
			where $M_{ij}^{(N)}$ and $m_{ij}^{(N)}$ defined in~\eqref{Eq:Mij}.
		\end{cor}
		
		\begin{proof}
			The result follows from applying Theorem~\ref{Thm:PD_quad} with $K=N$ to show that the quadratic function $\V$ in~\eqref{Eq:V_global} is positive definite.
		\end{proof}
		
		\begin{cor}[$\Delta\V$ negative semidefinite]\label{Cor:dV_NSD}
			Consider $\V$ in~\eqref{Eq:lyap} and define ${\Delta\V\defeq\V(\xi_{k+1})-\rho^2\,\V(\xi_k)}$. Then $\Delta\V$ is negative semidefinite for all values of $d\in\mathbb{N}$ if and only if there exists $\eta_{ij}\geq 0$ for $i,j\in\I_{N+1}$ such that
			\begin{align*}
			0 &\succeq \Delta V_N^{(N+1)} + \!\!\sum_{i,j\in\I_{N+1}}\!\! \eta_{ij}\,M_{ij}^{(N+1)} \\
			0 &\geq    \Delta v_N^{(N+1)} + \!\!\sum_{i,j\in\I_{N+1}}\!\! \eta_{ij}\,m_{ij}^{(N+1)}
			\end{align*}
			where $\Delta V_N^{(N+1)}$ and $\Delta v_N^{(N+1)}$ are defined in~\eqref{Eq:dV}.
		\end{cor}
		
		\begin{proof}
			The result follows from applying Theorem~\ref{Thm:PD_quad} with $K=N+1$ to show that the quadratic function $\Delta\V$ in~\eqref{Eq:dV_global} is negative semidefinite.
		\end{proof}
		
		Theorem~\ref{Thm:main} then follows from combining the results in Corollaries~\ref{Cor:V_PD} and~\ref{Cor:dV_NSD}. In particular, the inequalities in each corollary correspond to the constraints in the semidefinite program \sdp. If the problem is feasible, then $\V$ is positive definite and $\Delta\V$ is negative semidefinite at iteration $N$. Since this holds for any initial condition, we can apply the result for each $k\geq N$ to show that $\V$ is a valid Lyapunov function. On the other hand, if the problem is infeasible, then there exists no quadratic function of the form~\eqref{Eq:lyap} such that $\V$ is positive definite and $\Delta\V$ is negative semidefinite, so no valid quadratic Lyapunov function with state $\xi_k$ exists for the given rate $\rho$. This completes the proof of Thm.~\ref{Thm:main}. \qedhere

		\section{Extensions}\label{Sec:extensions}
		Our main result in Theorem~\ref{Thm:main} applies to methods of the form~\eqref{Eq:method} with fixed step-sizes applied to smooth strongly convex functions. Our framework, however, can be extended to many other scenarios, with or without tightness.		
		
		We now proceed with some examples of how our procedure of searching for Lyapunov functions can serve as a basis for the analysis of many \emph{exotic} algorithms. We provide two such examples: (i) the analysis of variants of GM and HBM involving subspace searches, and (ii) the analysis of a fast gradient scheme with scheduled restarts.
		
		\subsection{Exact Line Searches}
		In this section, we search for quadratic Lyapunov functions when it is possible to perform an exact line search. We illustrate the procedure on steepest descent
		\begin{equation*}
		\begin{aligned}
		\alpha &= \argmin_\alpha f(x_k-\alpha\, \df(x_k))\\
		x_{k+1}&=x_k-\alpha\, \df(x_k)
		\end{aligned}
		\end{equation*}
		and on a variant of HBM: 
		\begin{equation*}
		\begin{aligned}
		(\alpha,\beta)&= \argmin_{\alpha,\beta} f(x_k+\beta\, (x_k-x_{k-1})-\alpha\, \df(x_k))\\
		x_{k+1}&=x_k+\beta\, (x_k-x_{k-1})-\alpha\, \df(x_k)
		\end{aligned}
		\end{equation*}
		The detailed analyses can be found in the supplementary material, whereas the results are presented on Figure~\ref{fig:ELS_GM_HBM}.
		
		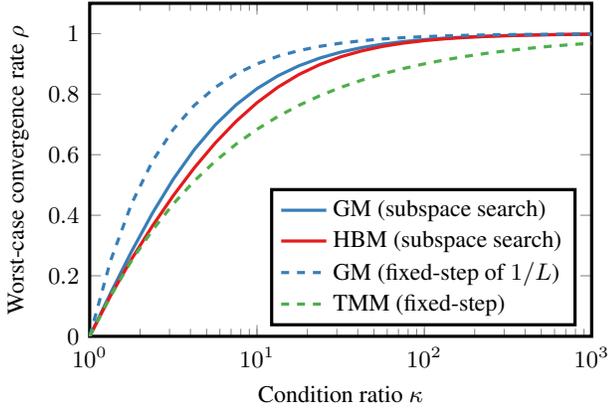
\begin{figure}[!ht]		
			\begin{center}
				\begin{tikzpicture}
				  \begin{semilogxaxis}[plotOptions,xmin=1e0,xmax=1e3,ymin=0,ymax=1.1,ytick={0,0.2,0.4,0.6,0.8,1},]
				    \addplot [colorGM] table [x index=0,y index=1,header=false] {data/data_SubspaceSearches.dat}; 
				    \addlegendentry{GM (subspace search)};
				    \addplot [colorHBM] table [x index=0,y index=2,header=false] {data/data_SubspaceSearches.dat}; 
				    \addlegendentry{HBM (subspace search)};
				    \addplot [colorGM,dashed] table [x index=0,y index=1,header=false] {data/data.dat};
				    \addlegendentry{GM (fixed-step of $1/L$)};
				    \addplot [colorTMM,dashed] table [x index=0,y index=4,header=false] {data/data.dat}; 
				    \addlegendentry{TMM (fixed-step)};
				  \end{semilogxaxis}
				\end{tikzpicture}
				\caption{\label{fig:ELS_GM_HBM} Convergence rates of GM and HBM with subspace searches. Note that the Gradient Method with exact line search matches the worst-case rate $\frac{\kappa-1}{\kappa+1}$ from~\cite{deKlerk2017}. For comparison, the rates of the Gradient Method with step-size $1/L$ and the Triple Momentum Method are also shown.}
			\end{center}
		\end{figure}

		\subsection{Scheduled Restarts}
		In this section, we apply the methodology to estimate the convergence rate of FGM with scheduled restarts; motivations for this kind of techniques can be found in e.g.,~\cite{o2015adaptive}. We present numerical guarantees obtained when using a version of FGM tailored for smooth convex minimization, which is restarted every $N$ iterations. This setting goes slightly beyond the fixed-step model presented in~\eqref{Eq:method}, as the step-size rules depend on the iteration counter.
		
		Define $\beta_0\defeq1$ and $\beta_{i+1}\defeq \frac{1+\sqrt{4\beta_i^2+1}}{2}$; we use the following iterative procedure 
		\begin{equation}\label{Eq:restarted_method}
			\begin{aligned}
			y^0_k,z^0_k &\leftarrow y^{N}_{k-1}\\
			z^{i+1}_k   &= y^i_k - \frac1L\df(y_k^i)\\
			y^{i+1}_k   &= z^{i+1}_k+\frac{\beta_i-1}{\beta_{i+1}} (z^{i+1}_k-z^{i}_k)
			\end{aligned}
			\end{equation}
			which does $N$ steps of the standard fast gradient method~\cite{nesterov1983method} before restarting. We study the convergence of this scheme using quadratic Lyapunov functions with states $(y_k^N-y_\star,\df(y_k^N),f(y_k^N)-f(y_\star))$. The derivations of the SDP for verifying $\V_{k+1}\leq \rho^{2N}\V_k$ (where~$\rho^{N}$ is the convergence rate of the inner loop) is similar to that of~\sdp\ (details in supplementary material).  Numerical results are provided in Figure~\ref{fig:restart}.
		
		\begin{figure}[!ht]		
			\begin{center}
				\begin{tikzpicture}
				  \begin{loglogaxis}[MyStyle2,xmin=1,xmax=1e3,ymin=2,ymax=1e3,legend pos=north west]
				    \addplot [colorFGM!40!white] table [x index=0,y index=4,header=false] {data/data_restarted_FGM.dat};
					\addlegendentry{$N=1$};
				    \addplot [colorFGM!60!white] table [x index=0,y index=3,header=false] {data/data_restarted_FGM.dat};
				    \addlegendentry{$N=5$};
				    \addplot [colorFGM!80!white] table [x index=0,y index=2,header=false] {data/data_restarted_FGM.dat};
				    \addlegendentry{$N=10$};
				    \addplot [colorFGM!100!white] table [x index=0,y index=1,header=false] {data/data_restarted_FGM.dat};
				    \addlegendentry{$N=20$};
				    \addplot [black,dashed] table [x index=0,y index=5,header=false] {data/data_restarted_FGM.dat};
				    \addplot [blue,dashed] table [x index=0,y index=1,header=false] {data/data_restarted_FGM_Opt_N.dat};
				  \end{loglogaxis}
				\end{tikzpicture}
				\caption{\label{fig:restart} Worst-case number of gradient evaluations to convergence $\mathcal{O}(-1/\log\rho)$ for different restart schedules $N$ (purple) along with the optimal restart schedule $N_\star = \arg\min_N\,\rho(N)$ (dashed blue). For comparison, we also plot the upper bound $\rho(N_\star)\leq\mathrm{exp}\left({\frac{-1}{e\sqrt{8\kappa}}}\right)$ (dashed black) from~\cite{o2015adaptive}. Results are not shown for small $\kappa$ due to numerical limitations of the SDP solvers.}
			\end{center}
		\end{figure}
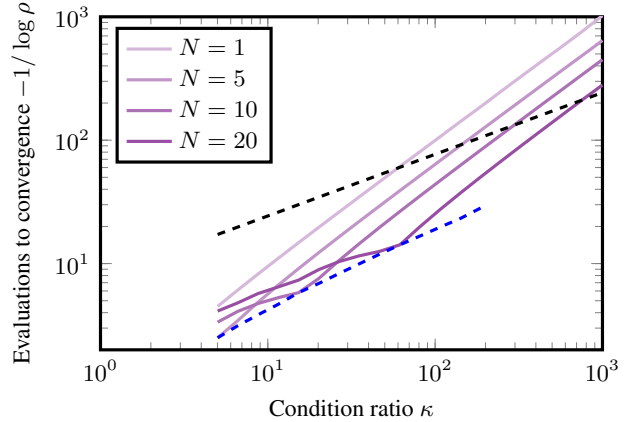

		\section{Conclusion}\label{Sec:conclusion}
		In this work, we studied first-order iterative fixed-step methods applied to smooth strongly convex functions. We presented a semidefinite formulation whose feasibility is both necessary and sufficient for the existence of a quadratic Lyapunov function.
		For smooth strongly convex minimization, restriction to quadratic Lyapunov functions is natural, as nonlinearities are exactly characterized by quadratic interpolation constraints. Using other tools such as sum-of-squares Lyapunov functions (see e.g.,~\citet{parrilo2000structured}) could be beneficial for more general algorithm and problem classes.

		This methodology unifies two previous approaches to worst-case analyses: performance estimation due to~\citet{drori2014} and integral quadratic constraints due to~\citet{lessard2016}. Moreover, this approach admits a large number of potential extensions, both in terms of classes of optimization problems and types of algorithms that can be analyzed (see e.g., extensions for performance estimation~\cite{taylor2017exact}). In particular, Lyapunov functions can be used to study sublinear convergence rates (see e.g.,~\citet{hu2017}), switched systems~\cite{lin2009} (e.g., for adaptive methods), noisy methods (see e.g.,~\citet{cyrus2018}), or continuous-time settings such as in~\cite{su2016differential}.
		
		\paragraph*{Code}
		The code used to implement \sdp\ and generate the figures in this paper is available at \url{https://github.com/QCGroup/quad-lyap-first-order}.
		
		\section*{Acknowledgments}
		{A.~Taylor was supported by the European Research Council (ERC) under the European Union's Horizon 2020 research and innovation program (grant agreement 724063). B.~Van~Scoy and L. Lessard were supported by the National Science Foundation under Grants No.~1656951 and~1750162.
		
		This work was launched during the \emph{LCCC Focus Period on Large-Scale and Distributed Optimization} organized by the Automatic Control Department of Lund University. The~authors thank the organizers.
		
		\nocite{boyd2004} 
		{ 
		\bibliography{references}
		\bibliographystyle{icml2018}
	}
		
		\appendix
		
		\section{Proof of Theorem~\ref{Thm:PD_quad} (Sampled Positive Definite Quadratics)}
		We begin by noting that $\sigma$ is positive definite if and only if the optimal value of the following problem is positive for any value of $\varepsilon>0$,
		\begin{align*}
		&\inf_\xi \ \sigma(\xi) \\
		&\text{s.t. } \left\{\begin{array}{l}
		(\x,\g,\f) \text{ as in~\eqref{Eq:xgf} generated by method~\eqref{Eq:method}} \\\qquad\text{ applied to $f\in\F$},\\
		\|\xi\|\geq \varepsilon.\end{array}\right.
		\end{align*}
		Next, we discretize the problem by replacing $f\in\F$ with the equivalent condition that the discrete set of points $\{(y_i,g_i,f_i)\}_{i\in\I_K}$ is $\F$-interpolable (recall that $\I_K=\{0,\hdots,K,\star\}$). Choosing a specific notion of distance for $\xi$, we can reformulate the previous statement as verifying that $p^{(d)}_\star(\varepsilon)>0$ for all $\varepsilon>0$ where
		\begin{align*}
		p_\star^{(d)}(\varepsilon) &\defeq \min_{\x,\g,\f} \ \sigma(\xi) \\
		&\ \text{s.t. } \left\{\begin{array}{l}
		\{(y_i,g_i,f_i)\}_{i\in\I_k}\text{ is $\F$-interpolable},\\
		(\x,\g,\f) \text{ as in~\eqref{Eq:xgf} generated by~\eqref{Eq:method}},\\
		\|\x\|^2+\|\g\|^2+\1^\tp \f = \varepsilon.\end{array}\right.
		\end{align*}
		Note that the last condition can equivalently be replaced by others (e.g., $\|\x\|^2=\varepsilon$ or $\|\g\|^2=\varepsilon$ or $\1^\tp\f=\varepsilon$), and that the optimal value of this problem is attained (using a short homogeneity argument with respect to $\varepsilon$). Using the necessary and sufficient conditions for the set of points $\{(y_i,g_i,f_i)\}_{i\in\I_k}$ to be $\F$-interpolable from Theorem~\ref{Thm:interp}, we have
		\begin{align*}
		p_\star^{(d)}(\varepsilon) &= \min_{\x,\g,\f} \ \sigma(\xi) \\
		&\ \text{s.t. } \left\{\begin{array}{l}
		\phi_{ij}\geq 0\text{ for all }i,j\in\I_K,\\
		(\x,\g,\f) \text{ as in~\eqref{Eq:xgf} generated by~\eqref{Eq:method}},\\
		\|\x\|^2+\|\g\|^2+\1^\tp \f = \varepsilon.\end{array}\right.
		\end{align*}
		Next, we define the Gram matrix $\G\in\mathbb{S}^{N+K+2}$ as ${\G\defeq\B^\tp\B}$ where
		\begin{align}\label{Eq:B}
		\B \defeq \bmat{x_{-N}-x_\star & \ldots & x_0-x_\star & g_0 & \ldots & g_K},
		\end{align}
		hence $\G$ is a standard Gram matrix containing all inner products between $x_i-x_\star$ for $i\in\{-N,\ldots,0\}$ and $g_i$ for $i\in\{0,\ldots,K\}$. Note that the quadratic $\sigma$ can be written as a function of the Gram matrix as
		\[
		\sigma(\G,\f) = \trace\bigl(Q \G\bigr) + q^\tp \f.
		\]
		Similarly, the interpolation conditions can also be reformulated with the Gram matrix as
		\[
		0 \leq \phi_{ij}(\G,\f) = \trace\bigl(M_{ij} \G\bigr) + m_{ij}^\tp \f
		\]
		where $M_{ij}$ and $m_{ij}$ are such that
		\[
		\phi_{ij} = \bmat{\x \\ \g}^\tp\!\!\! (M_{ij}\kron I_d)
		\bmat{\x \\ \g} + m_{ij}^\tp \f,
		\]
		for all $i,j\in\I_K$ (hence also $\star$ is in the index set).
		Therefore, we can reformulate the previous problem as the following rank-constrained semidefinite program:
		\begin{align*}
		p_\star^{(d)}(\varepsilon) =  &\min_{\G\in\mathbb{S}^{N+K+2},\f\in\real^{N+1}} \ \trace\bigl(Q\G\bigr) + q^\tp \f\\
		&\text{s.t. } \left\{\begin{array}{l}
		0\leq \trace\bigl(M_{ij} \G\bigr) + m_{ij}^\tp \f \for i,j\in\I,\\
		\trace(\G) + \1^\tp \f = \varepsilon,\\
		\f\geq 0,\\
		\G\succeq 0,\\
		\rank(\G)\leq d,\end{array}\right.
		\end{align*}
		where we remind the reader that $d$ is the dimension of the optimization problem~\eqref{Eq:prob}. Therefore, as discussed in~\cite{taylor2017smooth,taylor2017exact}, if we want a result that does not depend on the dimension (i.e, a $\sigma$ that is positive definite whatever the value of $d$), we have to verify that $p^{(\infty)}(\varepsilon)>0$ (which corresponds to assuming that $d\geq N+K+2$ since this is the dimension of $\G$). We then have the following semidefinite program:
		\begin{align*}
		p_\star^{(\infty)}(\varepsilon) =  &\min_{\G\in\mathbb{S}^{N+K+2},\f\in\real^{N+1}} \ \trace\bigl(Q\G\bigr) + q^\tp \f\\
		&\text{s.t. } \left\{\begin{array}{l}
		0\leq \trace\bigl(M_{ij} \G\bigr) + m_{ij}^\tp \f \for i,j\in\I,\\
		\trace(\G) + \1^\tp \f = \varepsilon,\\
		\f\geq 0,\\
		\G\succeq 0.\end{array}\right.
		\end{align*}
		A Slater point for this problem (i.e., a feasible point such that $\G\succ 0$) is obtained in the following section, so the optimal value of the primal problem is equal to the optimal value of the dual, which is given by
		\begin{align*}
		d^{(\infty)}(\varepsilon) \defeq  &\max_{\{\lambda_{ij}\},\nu} \ \nu\,\varepsilon \\
		&\text{s.t. } \left\{\begin{array}{l} \lambda_{ij}\geq 0\text{ for all }i,j\in\I,\\
		Q-\sum_{i,j\in\I} \lambda_{ij}\,M_{ij}\succeq \nu I_{N+K+2}, \\
		q-\sum_{i,j\in\I} \lambda_{ij}\,m_{ij}\geq \nu\1_{N+1}. \end{array}\right.\notag
		\end{align*}
		The theorem is then proved by noting the equivalence
		\[
		p^{(\infty)}(\varepsilon)>0, \ \forall \varepsilon>0 \ \
		\Longleftrightarrow \ \ d^{(\infty)}(\varepsilon)>0, \ \forall \varepsilon>0,
		\]
		where the last statement amounts to verifying that
		\[
		Q-\sum_{i,j\in\I} \lambda_{ij}\,M_{ij}\succ 0, \quad
		q-\sum_{i,j\in\I} \lambda_{ij}\,m_{ij}>0. \tag*{\qedhere}
		\]

		\section{Slater Point for Proof of Theorem~\ref{Thm:PD_quad}}
		In this section, we show how to construct a Slater point~\cite{boyd2004} for the primal semidefinite program in the proof of Theorem~\ref{Thm:PD_quad}. The construction is similar to Section 2.1.2 of~\cite{nesterov2004} and the proof of Theorem~6 in~\cite{taylor2017smooth}.
		
		Consider applying the first-order iterative fixed-step method~\eqref{Eq:method} with $\alpha\neq0$ and $\gamma_0\neq 0$ for $K$ iterations to the function $f(x)=\tfrac{1}{2} x^\tp H x$ where $H\in\mathbb{S}^d$ with $d\geq N+K+2$ is the positive definite tridiagonal matrix defined by
		\[
		[H]_{ij} = \begin{cases} 2 &\text{if }i=j \\ 1 &\text{if }|i-j|=1 \\ 0 & \text{otherwise} \end{cases}
		\]
		which has maximum eigenvalue $L=2+2\cos\bigl(\pi/(d+1)\bigr)$. Define the matrix $\B$ in~\eqref{Eq:B}. Using the initial condition $x_i=e_{N+1+i}$ for $i=-N,\ldots,0$, we will show that $\B$ is upper triangular with nonzero diagonal elements, and hence full rank.
		
		Since $\gamma_0\neq 0$, $y_0$ has a nonzero element corresponding to $e_{N+1}$. Then $g_0=H y_0$ has a nonzero element corresponding to $e_{N+2}$ due to the tridiagonal structure of $H$. Furthermore, $y_0$ may only have nonzero elements corresponding to $e_j$ for $j=1,\ldots,N+1$, so $g_0$ must have zero components corresponding to $e_j$ for all $j>N+2$.
		
		We now continue by induction. Assume that $g_k$ has a nonzero element corresponding to $e_{N+2+k}$ and zero elements corresponding to $e_i$ for all $i>N+2+k$. Since $\alpha\neq 0$, $x_{k+1}$ has a nonzero element corresponding to $e_{N+2+k}$ and zero elements corresponding to $e_i$ for all $i>N+2+k$. Then since $\gamma_0\neq 0$, $y_{k+1}$ also has the same structure. Due to the tridiagonal structure of $H$, $g_{k+1}$ then has a nonzero element corresponding to $e_{N+3+k}$ and zero elements corresponding to $e_i$ for all $i>N+3+k$. Therefore, by induction we have shown that for all $k\geq 0$, $g_k$ has a nonzero element corresponding to $e_{N+2+k}$ and zero elements corresponding to $e_i$ for all $i>N+2+k$. For $P$ to be upper triangular, we need $g_K$ to have dimension at least $N+2+K$. Thus, if $d\geq N+2+K$ where $K$ is the number of iterations, then $\B$ is upper triangular with nonzero entries on the diagonal, and therefore has full rank.
		
		In order to make the statement hold for general $\mu<L$, observe that the tridiagonal structure of $H$ is preserved under the operation
		\[
		\tilde{H} = \bigl(H-\lambda_\text{min}(H)\,I\bigr) \frac{L-\mu}{\lambda_\text{max}(H)-\lambda_\text{min}(H)} + \mu I
		\]
		where $\mu I\preceq \tilde{H}\preceq L I$.
		
		Since $\B$ has full rank, the Gram matrix $\G = \B^\tp \B\succ 0$ is positive definite. Therefore, the primal semidefinite program satisfies Slater's condition.

		\section{Steepest Descent}
		In this section, we show a similar formulation as~\sdp\ for steepest descent. In this case, the analysis was not \emph{a priori} guaranteed to be tight, due to the line search conditions. In order to encode the line search, we use the corresponding optimality conditions, as in~\cite{deKlerk2017}:
		\begin{equation}\label{eq:ELS}
		\begin{aligned}
		\langle x_{k+1}-x_k,\ g_{k+1}\rangle&=0,\\
		\langle g_k,\ g_{k+1}\rangle&=0,
		\end{aligned}			
		\end{equation}with $g_k=\nabla f(x_k)$. For the Lyapunov function structure, we choose the following
		\begin{align*}
		V(\xi_k)=\bmat{x_k-x_\star\\g_k}(P\otimes I_d)\bmat{x_k-x_\star\\g_k}^\tp+p (f_k-f_\star).
		\end{align*}
		In order to develop a SDP formulation for this problem we follow the same steps as for~\sdp, starting with {\bf Step 1} (see Section~\ref{Sec:sdp}): we define the following row vectors in~$\real^{2}$ 
		\[\bay_0^{(0)}\defeq\e_1^\tp,\, \bax_0^{(0)}\defeq \e_1^\tp,\, \bag_0^{(0)}\defeq\e_2^\tp,\]
		and $\bay_\star^{(0)}=\bax_\star^{(0)}=\bag_\star^{(0)}\defeq\0^\tp$, along with the scalars $\baf_0^{(0)}\defeq1$ and $\baf_\star^{(0)}\defeq0$. In addition, we use the following vectors in $\real^{4}$ 
		\begin{align*}
		&\bax_0^{(1)}\defeq\e_1^\tp,\, \bax_1^{(1)}\defeq\e_2^\tp,\\ &\bay_0^{(1)}\defeq\e_2^\tp,\,\bay_1^{(1)}\defeq\e_2^\tp,\\ &\bag_0^{(1)}=\e_3^\tp,\, \bag_1^{(1)}\defeq\e_4^\tp,
		\end{align*}
		and $\bay_\star^{(1)}=\bax_\star^{(1)}=\bag_\star^{(1)}\defeq\0^\tp$, along with $\baf^{(1)}_0,\baf^{(1)}_1,\baf^{(1)}_\star\in\real^2$ such that $\baf^{(1)}_0\defeq\e_1^\tp$, $\baf^{(1)}_1\defeq\e_2^\tp$ and $\baf^{(1)}_\star\defeq\0^\tp$.
		
		Because of the algorithm, {\bf Step 2} is slightly different as before; we encode the line search constraints~\eqref{eq:ELS} using
		\begin{equation*}
		\begin{aligned}
		A_{1}&=\bmat{\bax_0^{(1)}\\\bax_1^{(1)}\\\bag^{(1)}_1}^\tp\bmat{0 & 0 & -1\\ 0 &0 &1\\ -1& 1& 0}\bmat{\bax_0^{(1)}\\\bax_1^{(1)}\\\bag^{(1)}_1},\\
		A_{2}&= \bmat{\bag_0^{(1)}\\ \bag_1^{(1)}}^\tp \bmat{0 & 1 \\ 1 & 0}\bmat{\bag_0^{(1)}\\ \bag_1^{(1)}}.
		\end{aligned}
		\end{equation*}
		
		The subsequent steps ({\bf Step 3} and {\bf Step 4}) are exactly the same as in Section~\ref{Sec:sdp}. We finally obtain a slightly modified version of the feasibility problem~\sdp:
		\begin{align*}
		\smash[b]{\feasible_{\substack{P\in\mathbb{S}^{2}\\p\in\real^{1}\\\nu_1,\nu_2\in\real\\ \{\lambda_{ij}\}\\ \{\eta_{ij}\}}}}
		& \quad 0 \prec V_0^{(0)} - \sum_{i,j\in\I_0} \lambda_{ij}\,M_{ij}^{(0)} \\
		& \quad 0 < v_0^{(0)} - \sum_{i,j\in\I_0} \lambda_{ij}\,m_{ij}^{(0)} \\
		& \quad 0 \succeq \Delta V_0^{(1)} + \!\!\!\sum_{i,j\in\I_{1}} \!\!\eta_{ij}\,M_{ij}^{(1)}+\sum_{i=1}^2\nu_i \,A_i\\
		& \quad 0 \geq \Delta v_0^{(1)} + \!\!\!\sum_{i,j\in\I_{1}} \!\! \eta_{ij}\,m_{ij}^{(1)} \\
		& \quad 0 \leq \lambda_{ij} \for i,j\in\I_0 \\
		& \quad 0 \leq \eta_{ij} \for i,j\in\I_{1}
		\end{align*}
		with $\I_0\defeq\{0,\star\}$ and $\I_1\defeq\{0,1,\star\}$.

		\section{SDP for HBM with Subspace Searches}
		We follow the steps of the previous section for steepest descent; we only make the following adaptations: (i) we look for a quadratic Lyapunov function with the states
		\begin{align*}V(\xi_k)=\bmat{x_{k}-x_\star\\x_{k-1}-x_\star\\g_{k}\\g_{k-1}}&(P\otimes I_d)\bmat{x_{k}-x_\star\\x_{k-1}-x_\star\\g_{k}\\g_{k-1}}^\tp\\&+p^\tp \bmat{f_{k}-f_\star\\f_{k-1}-f_\star},
		\end{align*}
		(ii) we adapt the initialization ({\bf Step 1}), (iii) adapt the line search conditions ({\bf Step 2}) and (iv) obtain a slightly modified version of the SDP.
		
		For (ii), we adapt the initialization procedure ({\bf Step 1}) as follows. We define the following row vectors of $\real^4$:
		\begin{align*}
		&\bax_{0}^{(1)}\defeq\e_1^\tp,\, \bax_{1}^{(1)}\defeq\e_2^\tp,\,\\& \bay_{0}^{(1)}\defeq\e_1^\tp,\, \bay_{1}^{(1)}\defeq\e_2^\tp,\,\\ & \bag_{0}^{(1)}\defeq\e_3^\tp,\, \bag_{1}^{(1)}\defeq\e_4^\tp,
		\end{align*}
		along with $\bay_\star^{(1)}=\bax_\star^{(1)}=\bag_\star^{(1)}\defeq\0^\tp$, and the following in~$\real^2$: $\baf^{(1)}_0\defeq\e_1^\tp$, $\baf^{(1)}_1\defeq\e_2^\tp$ and $\baf^{(1)}_\star\defeq\0^\tp$. We also define the following row vectors in $\real^7$:
		\begin{align*}
		\bax_{-1}^{(2)}\defeq\ \e_1^\tp,\, \bax_{0}^{(2)}\defeq\e_2^\tp&,\, \bax_{1}^{(2)}\defeq\e_3^\tp,\, \bax_{2}^{(2)}\defeq\e_4^\tp,\\ 
		\bay_{0}^{(2)}\defeq\bax_{0}^{(2)},\, \bay_{1}^{(2)}&\defeq\bax_{1}^{(2)},\, \bay_{2}^{(2)}\defeq\bax_{2}^{(2)},\\
		\bag_{0}^{(2)}\defeq\e_5^\tp,\, \bag_{1}^{(2)}&\defeq\e_6^\tp,\, \bag_{2}^{(2)}\defeq\e_7^\tp,
		\end{align*}
		along with $y_\star^{(2)}=\bax_\star^{(2)}=\bag_\star^{(2)}\defeq\0^\tp$ and the vectors of~$\real^3$: $\baf^{(2)}_0\defeq\e_1^\tp$, $\baf^{(2)}_1\defeq\e_2^\tp$, $\baf^{(2)}_2\defeq\e_3^\tp$ and $\baf^{(2)}_\star\defeq\0^\tp$.
		
		Now for (iii) (or {\bf Step 2}), optimality of the search conditions can be 
		\begin{equation*}
		\begin{aligned}
		\langle x_{k+1}-x_k,\ g_{k+1}\rangle&=0,\\
		\langle x_k-x_{k-1},\ g_{k+1}\rangle&=0,\\
		\langle g_k; g_{k+1}\rangle&=0,
		\end{aligned}	
		\end{equation*}	
		which we can formulate in matrix form for $k\in\{0,1\}$:
		\begin{equation*}
		\begin{aligned}
			A_{1+k}&=\bmat{\bax_k^{(2)}\\\bax_{k+1}^{(2)}\\\bag^{(2)}_{k+1}}^\tp \bmat{0 & 0 & -1\\ 0 &0 &1\\ -1& 1& 0}\bmat{\bax_k^{(2)}\\\bax_{k+1}^{(2)}\\\bag^{(2)}_{k+1}}\\
			A_{3+k}&=\bmat{\bax_{k-1}^{(2)}\\\bax_{k}^{(2)}\\\bag^{(2)}_{k+1}}^\tp\bmat{0 & 0 & -1\\ 0 &0 &1\\ -1& 1& 0}\bmat{\bax_{k-1}^{(2)}\\\bax_{k}^{(2)}\\\bag^{(2)}_{k+1}}\\
			A_{5+k}&=\bmat{\bag_k^{(2)}\\\bag_{k+1}^{(2)}}^\tp\bmat{0 & 1\\ 1& 0}\bmat{\bag_k^{(2)}\\\bag_{k+1}^{(2)}}
			\end{aligned}
		\end{equation*}
		
		The subsequent steps ({\bf Step 3} and {\bf Step 4}) are exactly the same as in Section~\ref{Sec:sdp}. We finally obtain a slightly modified version of the feasibility problem~\sdp:
		\begin{align*}
		\smash[b]{\feasible_{\substack{P\in\mathbb{S}^{2}\\p\in\real^{1}\\\nu_1,\hdots,\nu_6\in\real\\ \{\lambda_{ij}\}\\ \{\eta_{ij}\}}}}
		& \quad 0 \prec V_1^{(1)} - \sum_{i,j\in\I_1} \lambda_{ij}\,M_{ij}^{(1)} \\
		& \quad 0 < v_1^{(1)} - \sum_{i,j\in\I_1} \lambda_{ij}\,m_{ij}^{(1)} \\
		& \quad 0 \succeq \Delta V_1^{(2)} + \!\!\!\sum_{i,j\in\I_{2}} \!\!\eta_{ij}\,M_{ij}^{(2)}+\sum_{i=1}^6\nu_i \,A_i\\
		& \quad 0 \geq \Delta v_1^{(2)} + \!\!\!\sum_{i,j\in\I_{2}} \!\! \eta_{ij}\,m_{ij}^{(2)} \\
		& \quad 0 \leq \lambda_{ij} \for i,j\in\I_1 \\
		& \quad 0 \leq \eta_{ij} \for i,j\in\I_{2}
		\end{align*}
		with $\I_1\defeq\{0,1,\star\}$ and $\I_2\defeq\{0,1,2,\star\}$. The corresponding results are presented on Figure~\ref{fig:ELS_GM_HBM}.
		
		\section{SDP for FGM with Scheduled Restarts}
		This setting goes slightly beyond the fixed-step model presented in~\eqref{Eq:method}, as the step sizes depend on the iteration. We study the algorithm described by~\eqref{Eq:restarted_method}, which does $N$ steps of the standard fast gradient method~\cite{nesterov1983method} before restarting. We study the convergence of this scheme using quadratic Lyapunov functions of the form
		\begin{equation}\label{eq:lyap_with_restart}
		\bmat{y_k^N-x_\star \\ \df(y_k^N)}^\tp (P\otimes I_d) \bmat{y_k^N-x_\star\\ \df(y_k^N)} +p\, [f(y_k^N)-f(x_\star)].
		\end{equation}
		Let us perform similar steps as for~\sdp\ for constructing the corresponding SDP. We start with the initialization procedure {\bf Step 1}. Let us define the following row vectors in~$\real^2$:
		\[\bay_0^{(1)}=\bax_0^{(1)}\defeq\e_1^\tp,\, \bag_0^{(1)}\defeq\e_2^\tp, \]
		and $\bax_\star^{(1)}=\bag_\star^{(1)}\defeq\0^\tp$, along with the scalars $f_0^{(1)}\defeq1$ and $f_\star^{(1)}\defeq0$. In addition, we define the row vectors of $\real^{N+2}$:
		\[\bax_0^{(N+1)}\defeq\e_1^\tp,\, \bag^{(N+1)}_k\defeq\e_{2+k}^\tp, \]
		for $k=0,\hdots,N$, along with $\bax_\star^{(N+1)}=\bag^{(N+1)}_\star\defeq\0^\tp$ and the row vectors $\baf^{(N+1)}_k\in\real^{N+1}$ defined as $\baf^{(N+1)}_k\defeq\e_{1+k}^\tp$ and $\baf^{(N+1)}_\star\defeq\0^\tp$.
		
		{\bf Step 2} Apply one complete loop of the algorithm as follows: for $k=0,\hdots,N-1$ define the sequence of row vectors:
		\begin{equation*}
		\begin{aligned}
		\baz_{k+1}^{(N+1)}   &= \bay_k^{(N+1)} - \frac1L \bag^{(N+1)}_k,\\
		\bay_{k+1}^{(N+1)}   &= \baz^{(N+1)}_k+\frac{\beta_k-1}{\beta_{k+1}} (\baz^{(N+1)}_k- \baz^{(N+1)}_k),
		\end{aligned}
		\end{equation*}
		with $\beta_0\defeq1$ and $\beta_{k+1}\defeq \frac{1+\sqrt{4\beta_k^2+1}}{2}$. For complying with the notations of the paper, we define the sequence
		\[x_k^{(K)}\defeq y_k^{(K)}\]
		for $k=0,\hdots,N$ and $K\in\{1,N+1\}$. Then, using the sets $\I_{1}\defeq\{0,\star\}$ and $\I_{N+1}\defeq\{0,\hdots,N,\star\}$, the other stages follow from the same lines as {\bf Step 3}, and {\bf Step 4} with the slight modification of the expression for the rate
		\begin{align*}
		\Delta v_k^{(N+1)} &\defeq v_{k+1}^{(N+1)} - \rho^{2N}\,v_k^{(N+1)}, \\
		\Delta V_k^{(N+1)} &\defeq V_{k+1}^{(N+1)} - \rho^{2N}\,V_k^{(N+1)},
		\end{align*}
		and {\bf Step 5} follows as in Section~\ref{Sec:sdp}:
		\begin{align*}
		\smash[b]{\feasible_{\substack{P\in\mathbb{S}^{2(N+1)}\\p\in\real^{N+1}\\ \{\lambda_{ij}\}\\ \{\eta_{ij}\}}}}
		& \quad 0 \prec V_{1}^{(1)} - \sum_{i,j\in\I_1} \lambda_{ij}\,M_{ij}^{(1)} \\
		& \quad 0 < v_{1}^{(1)} - \sum_{i,j\in\I_1} \lambda_{ij}\,m_{ij}^{(1)} \\
		& \quad 0 \succeq \Delta V_N^{(N+1)} + \!\!\!\sum_{i,j\in\I_{N+1}} \!\!\eta_{ij}\,M_{ij}^{(N+1)} \\
		& \quad 0 \geq \Delta v_N^{(N+1)} + \!\!\!\sum_{i,j\in\I_{N+1}} \!\! \eta_{ij}\,m_{ij}^{(N+1)} \\
		& \quad 0 \leq \lambda_{ij} \for i,j\in\I_1 \\
		& \quad 0 \leq \eta_{ij} \for i,j\in\I_{N+1}
		\end{align*}
		(note that the sets $\I_1$ and $\I_{N+1}$ should use the definitions of this section). Numerical results are available in Figure~\ref{fig:restart}.
		
\end{document}